\documentclass[]{amsart}

\usepackage{graphicx}
\usepackage{ulem}
\usepackage{bm}
\usepackage{color}
\usepackage{amssymb, amsmath, amsthm}
\usepackage{bm}
\theoremstyle{plain}
\usepackage{graphicx}
\usepackage[numbers]{natbib}

\newtheorem{lemma}{Lemma}[section]

\newtheorem{theorem}[lemma]{Theorem}

\newtheorem{prop}[lemma]{Proposition}
\newtheorem{exam}[lemma]{\normalfont \scshape
 Example}
\newtheorem{rem}[lemma]{\normalfont \scshape Remark}

\newcommand{\R}{\mathbb{R}}
\newcommand{\N}{\mathbb{N}}
\newcommand{\norm}[1]{\left\Vert#1\right\Vert}
\newcommand{\abs}[1]{\left\vert#1\right\vert}
\newcommand{\set}[1]{\left\{#1\right\}}

\newcommand{\bfx}{\bm{x}}
\newcommand{\bfzero}{\bm{0}}

\newcommand{\bfone}{\bm{1}}

\newcommand{\bfC}{\bm{C}}

\newcommand{\bfU}{\bm{U}}

\newcommand{\bfV}{\bm{V}}

\newcommand{\bfX}{\bm{X}}
\newcommand{\bfY}{\bm{Y}}
\newcommand{\bfy}{\bm{y}}
\newcommand{\bfZ}{\bm{Z}}

\newcommand{\bfeta}{\bm{\eta}}

\newcommand{\bfxi}{\bm{\xi}}

\newcommand{\barE}{\bar E^-[0,1]}

\allowdisplaybreaks[4]

\begin{document}

\author[M. Falk and M. Hofmann]{Michael Falk and Martin Hofmann}

\address{\mbox{ } \newline University of W\"{u}rzburg \newline
Institute of Mathematics\newline  Emil-Fischer-Str. 30\newline 97074 W\"{u}rzburg\newline Germany\newline hofmann.martin@mathematik.uni-wuerzburg.de}

\title[Sojourn Times and the Fragility Index]{Sojourn Times and the Fragility Index\\
        {\tiny \normalfont to appear in ``Stochastic Processes and their Applications''}}%

\begin{abstract}
We investigate the sojourn time above a high threshold of a
continuous stochastic process $\bfY=(Y_t)_{t\in[0,1]}$. It turns out that the limit, as the threshold increases, of the expected sojourn time given that it is positive, exists if the copula process corresponding to $\bfY$ is in the functional domain of attraction of a max-stable process. This limit coincides with the limit of the fragility
index corresponding to $(Y_{i/n})_{1\le i\le n}$ as $n$ and the
threshold increase.

If the process is in a certain neighborhood of a generalized Pareto process, then we can replace the constant threshold by a general threshold function and we can compute the asymptotic sojourn time distribution. A max-stable process is a prominent example.
Given that there is an exceedance at $t_0$ above the
threshold, we can also compute the asymptotic distribution of the
excursion time, which the process spends above the threshold function.
\end{abstract}



\maketitle

\section{Introduction}

Let $\bfY=(Y_t)_{t\in[0,1]}$ be a stochastic process with
continuous sample paths, i.e., $\bfY$ realizes in $C[0,1]$, and identical
continuous marginal distribution functions (df) $F$, say. We investigate
in this paper the  sojourn time of $\bfY$
above a threshold $s$
\[
S(s):=\int_0^1 1(Y_t>s)\,dt,
\]
under the condition that there is an exceedance, i.e.,
$S(s)>0$. Sojourn times of stochastic processes have been
extensively studied in the literature, with emphasis on Gaussian
processes and Markov random fields, we refer to \citet{berm92}
and the literature given therein. A more general approach is the excursion random measure as investigated by \citet{hsile98} for stationary processes. It is defined on sets $E\subset \R\times (0,\infty)$ as the time which the process (suitably standardized) will spend in $E$. Different to that,  we will investigate the sojourn time under the
condition that the copula process $\bfC:=(F(Y_t))_{t\in[0,1]}$
corresponding to $\bfY$ is in the functional domain of attraction
of a max-stable process $\bfeta$, say.

Denote by $N_s:=\sum_{i=1}^n 1_{(s,\infty)}(Y_{i/n})$ the number
of exceedances among $(Y_{i/n})_{1\le i\le
n}$ above the threshold $s$. The \textit{fragility index} (FI) corresponding to $(Y_{i/n})_{1\le
i\le n}$ is defined as the  asymptotic expectation of the number
of exceedances given that there is at least one exceedance:
\[
 FI:=\lim_{s\nearrow\omega(F)} E(N_s \mid N_s>0),
\]
where $\omega(F):=\sup\set{t\in\R:\,F(t)<1}$. The FI was introduced in \citet{gelhv07} to measure the stability
of a stochastic system. The system is called stable if $FI =1$,
otherwise it is called  fragile. The collapse of a bank,
symbolized by an exceedance, would be a typical example,
illustrating the FI as a measure of joint stability among a
portfolio of banks.

It turns out that the limit, as the threshold increases, of the expected sojourn time given that it is positive, exists if the copula process corresponding to $\bfY$ is in the functional domain of attraction of a max-stable process. This limit coincides with the limit of the FI corresponding to $(Y_{i/n})_{1\le i\le n}$ as $n$ and the
threshold increase.

For such processes, which are in a certain neighborhood of a generalized Pareto process (see Example \ref{exam:generalized_Pareto_process}), we can replace the constant threshold by a threshold function and we can compute the asymptotic sojourn time distribution above a high threshold function. A max-stable process is a prominent example. Given that there is an exceedance $Y_{t_0}>s$  above the
threshold $s$ at $t_0$, we can also compute the asymptotic distribution of the
remaining excursion time, that the process spends above the threshold function without cease.

This paper is organized as follows. In Section
\ref{sec:extreme_value_processes} we recall some mathematical
framework from functional extreme value theory and provide basic
definitions and tools. In particular we consider a functional
domain of attraction approach for stochastic processes, which is
more general than the usual one based on weak convergence.
In Section \ref{sec:domain_of_attraction_for_copula_processes} we
apply the framework from Section \ref{sec:extreme_value_processes}
to copula processes and derive characterizations of the domain of
attraction condition for copula processes.
In Section \ref{sec:sojourn_times_and_the_fragility_index} we use the results
from Section \ref{sec:domain_of_attraction_for_copula_processes}
to compute the limit $\lim_{s\nearrow\omega(F)}E(S(s)\mid S(s)>0)$ as the
threshold $s$ increases of the mean sojourn time, conditional on
the assumption that it is positive. We show that this limit
coincides with the FI. Our tools enable also the computation of the expected shortfall.
In Section \ref{sec: sojourn_time_distribution} we replace the constant threshold by a threshold function and we compute the limit distribution of the sojourn time for those processes, which are in a certain neighborhood of a generalized Pareto process.
Given that there is an exceedance at $t_0$, we compute in Section \ref{sec:cluster_length} the asymptotic distribution of the remaining excursion time that the process spends above a high threshold function.

To improve the readability of this paper we use bold face such as
$\bfxi$, $\bfY$ for stochastic processes and default font $f$,
$a_n$ etc. for non stochastic functions. Operations on functions
such as $\bfxi<a$ or $(\bfxi-b_n)/a_n$ are meant componentwise. The usual abbreviations \textit{df, fidis, iid} and \textit{rv} for the terms \textit{distribution function, finite dimensional distributions, independent and identically distributed} and  \textit{random variable}, respectively,  are used.

\section{Definitions and Preliminaries}

\subsection{Max-stable Processes and the Functional $D$-Norm}\label{sec:extreme_value_processes}

A \textit{max-stable process} (MSP)
$\bfxi=\left(\xi_t\right)_{t\in[0,1]}$  with realizations in $C[0,1]:=\{f:[0,1]\to\R:\ f \textrm{ continous}\}$, equipped with the sup-norm $\norm f_\infty=\sup_{t\in[0,1]}\abs{f(t)}$, is a
stochastic process with the characteristic property that its
distribution is max-stable, i.e., $\bfxi$ has the same distribution
as $\max_{1\leq i\leq n}(\bfxi_i-b_n)/a_n$ for independent copies
$\bfxi_1,\bfxi_2,\dots$ of $\bfxi$ and some $a_n,b_n\in C[0,1],\,
a_n>0$, $n\in\N$ (cf. \citet{dehaf06}).

We call a process $\bfeta$ with values in $C^-[0,1]:=\set{f\in C[0,1]:\,f < 0}$
a \textit{standard} MSP, if it is a MSP with standard negative
exponential (one-dimensional) margins, $P(\eta_t\le x)=\exp(x)$,
$x\le 0$, $t\in[0,1]$.

In what follows  $\bar C^-[0,1]:=\set{f\in
C[0,1]:\,f\le 0}$ denotes the set of all continuous function on $[0,1]$
that do not attain positive values.

The following characterization is essentially due to \citet{ginhv90}; we refer also to \citet{aulfaho11}.

\begin{prop}\label{prop:characterization_of_EP}
A process $\bfeta$ with realizations in $C^-[0,1]$ is a standard MSP if, and only if there exists a number $m\ge 1$ and a stochastic process $\bfZ$ in $\bar C^+[0,1]:=\set{f\in C[0,1]:\,f\ge 0}$ with the properties
\begin{equation}\label{eqn:properties_of_generator}
\max_{t\in[0,1]}Z_t = m,\qquad E(Z_t)=1,\quad t\in[0,1],
\end{equation}
such that for compact subsets $K_1,\dots,K_d$ of $[0,1]$ and $x_1,\dots,x_d\le 0$, $d\in\N$,
\begin{equation}\label{eqn:finite_distribution_of_EP}
P(\eta_t\le x_j,\,t\in K_j,\, 1\le j\le d)=\exp\left(-E\left(\max_{1\le j\le d} \left(\abs{x_j}\max_{t\in K_j}Z_t\right)\right)\right).
\end{equation}

Conversely, every stochastic process $\bfZ$ with realizations in $\bar C^+[0,1]$ satisfying \eqref{eqn:properties_of_generator} gives rise to a standard MSP. The connection is via \eqref{eqn:finite_distribution_of_EP}. We call $\bfZ$ \textit{generator} of $\bfeta$.
\end{prop}

According to \citet[Corollary 9.4.5]{dehaf06} the condition $\max_{t\in[0,1]}Z_t = m$ in \eqref{eqn:properties_of_generator}  can be replaced by the condition $E\left(\max_{t\in[0,1]}Z_t\right)<\infty$. The number $m=E\left(\max_{t\in[0,1]}Z_t\right)$ is uniquely determined, see Remark \ref{rem:the_generator_constant_is_uniquely_determined}. Therefore, we call $m$ the \textit{generator constant} of $\bfeta$.

The preceding characterization implies in particular that the fidis of $\bfeta$ are multivariate negative EVD with standard negative exponential margins: We have for $0\le t_1<t_2\dots<t_d\le 1$
\begin{equation}\label{eqn:definition_of_D-norm}
-\log(G_{t_1,\dots,t_d})(\bfx)=E\left(\max_{1\le i\le
d}(\abs{x_i}Z_{t_i})\right)
=:\norm{\bfx}_{D_{t_1,\dots,t_d}},\ \bfx\le\bfzero\in\R^d,
\end{equation}
where $\norm{\cdot}_{D_{t_1,\dots,t_d}}$ is a \textit{$D$-norm} on $\R^d$ (cf. \citet{fahure10}).

Let $E[0,1]$ be the set of all bounded real-valued functions on $[0,1]$ which are discontinuous at a finite set of points. Moreover, denote by $\barE$ the set of those functions in $E[0,1]$ which do not attain positive values.

For a generator process $\bfZ$ in $\bar C^+[0,1]$ as in
Proposition \ref{prop:characterization_of_EP} and all $f\in E[0,1]$ set
\[
\norm{f}_D:=E\left(\sup_{t\in [0,1]} \left(\abs{f(t)}Z_t\right)\right).
\]
Obviously, $\norm{\cdot}_D$  defines a norm on $E[0,1]$, called a
\textit{$D$-norm with generator $\bfZ$}; see \citet{aulfaho11} for further details.

The following result is established in \citet{aulfaho11}.

\begin{lemma}\label{prop:distribution_function_of_standard_EP}
Let $\bfeta$ be a standard MSP with generator $\bfZ$. Then we have for each $ f\in \bar E^-[0,1]$
\begin{equation}\label{eq:distribution_function_of_standard_EP}
P(\bfeta\le f) = \exp\left(-\norm f_D\right)= \exp\left(-E\left(\sup_{t\in [0,1]}\left(\abs{f(t)}Z_t\right)\right)\right).
\end{equation}

Conversely, if there is some $\bfZ$ with properties \eqref{eqn:properties_of_generator} and some $\bfeta\in C^-[0,1]$ which satisfies \eqref{eq:distribution_function_of_standard_EP}, then $\bfeta$ is standard max-stable with generator $\bfZ$.
\end{lemma}

The representation $ P(\bfeta\le f) = \exp\left(-\norm f_D\right)$, $f\in \bar
C^-[0,1]$, of a standard MSP is in complete accordance with the df of a
multivariate EVD with standard negative exponential margins via a
$D$-norm on $\R^d$ as developed in \citet[Section 4.4]{fahure10}.

Note that for $d\in\N$, $0\le t_1<\dots<t_d\le 1$ and $\bfx=(x_1,\dots,x_d)\in(-\infty,0]^d$,  the function
\[
f(t)=\sum_{i=1}^d x_i 1_{\set{t_i}}(t),\qquad t \in[0,1],
\]
is an element of $\bar E^-[0,1]$ with the property
\begin{equation*}
P(\bfeta\le f)= \exp\left(-\norm{\bfx}_{D_{t_1,\dots,t_d}}\right).
\end{equation*}
So representation \eqref{eq:distribution_function_of_standard_EP} incorporates all fidis of $\bfeta$. This is one of the reasons, why we favor a MSP $\bfeta$ with standard negative exponential margins, whereas \citet{dehaf06}, for instance, consider a continuous MSP $\bfxi=(\xi_t)_{t\in[0,1]}$ with standard Fr\'{e}chet margins $P(\xi_t\le x)=\exp\left(-x^{-1}\right)$, $x>0$, called \textit{simple} MSP. Actually, these are dual approaches, as we have
\[
\bfxi=-\frac 1{\bfeta}\mbox{  and  }\bfeta=-\frac 1 \bfxi,
\]
taken pointwise (see \citet{aulfaho11}). A simple MSP satisfies for $g:[0,1]\to(0,\infty)$ with $\tilde f:=-1/g \in\bar E^-[0,1]$, consequently,
\[
P(\bfxi\le g)=P\left(\bfeta\le -\frac 1 g\right)=\exp\left(-\norm{\frac 1 g}_D\right),
\]
but, different to $\bfeta$ we do not obtain the fidis of $\bfxi$ by a suitable choice of $g$.

Just like in the uni- or multivariate case, we might consider
\begin{equation*}\label{eqn:distribution_function_of_stochastic_process}
 H(f):=P(\bfY\le f),\qquad f\in \barE,
\end{equation*}
as the  df  of a stochastic process $\bfY$ in $\bar C^-[0,1]$.

\subsection{Functional Domain of Attraction}

According to \citet{aulfaho11} we say that a stochastic process $\bfY$ in $C[0,1]$ is \textit{in the functional domain of attraction of a standard MSP}   $\bfeta$,
denoted by $\bfY\in \mathcal D(\bfeta)$, if there are functions
$a_n\in C^+[0,1]$, $b_n\in C[0,1]$, $n\in\N$, such that
\begin{equation}\label{cond:definition_of_domain_of_attraction}
\lim_{n\to\infty}P\left(\frac{\bfY-b_n}{a_n}\le f\right)^n = P(\bfeta\le f)=\exp\left(-\norm f_D\right)      
\end{equation}
for any $f\in\barE$. This is equivalent to
\begin{equation}\label{cond:_equivalent_definition_of_domain_of_attraction}
 \lim_{n\to\infty} P\left(\max_{1\le i\le n}\frac{\bfY_i-b_n}{a_n}\le f\right)=P(\bfeta\le f)       
\end{equation}
for any $f\in\barE$, where $\bfY_1,\bfY_2,\dots$ are independent copies of $\bfY$.

There should be no risk of confusion with the notation of domain
of attraction in the sense of weak convergence of stochastic
processes as investigated in \citet{dehal01}. But to
distinguish between these two approaches we will consistently
speak of \textit{functional} domain of attraction in this paper,
when the above definition is meant. Actually, this definition of
domain of attraction is less restrictive as the next lemma shows; it is established in \citet{aulfaho11}.

\begin{lemma}\label{lem:weak_convergence_implies_domain_of_attraction}
Suppose that $\bfY$ is a continuous process in $\bar C^-[0,1]$.
If the sequence of continuous processes $\bfX_n:=\max_{1\le i\le
n}\left(\left(\bfY_i-b_n)/a_n\right)\right)$ converges weakly in
$\bar C^-[0,1]$, equipped with the sup-norm $\norm\cdot_\infty$,
to the standard MSP $\bfeta$, then $\bfY\in\mathcal D(\bfeta)$ in the sense of condition \eqref{cond:definition_of_domain_of_attraction}.
\end{lemma}

Note that the reverse implication in the preceding does not hold, i.e., convergence in the sense of condition \eqref{cond:definition_of_domain_of_attraction} is strictly weaker than weak convergence in $C[0,1]$.  One can also show that \eqref{cond:definition_of_domain_of_attraction} implies hypoconvergence of the normalized maximum process in the sense of \citet[Chapter 5, Section 3.1]{mol05}.

\subsection{Domain of Attraction for Copula Processes}\label{sec:domain_of_attraction_for_copula_processes}
 The sojourn time distribution of a stochastic process $(Y_t)_{t\in[0,1]}$ with identical continuous univariate marginal df $F$  does not depend on this marginal df but on the corresponding copula process. This is immediate from the equality $\int_0^11(Y_t > s)\,dt = \int_0^1 1(U_t > F(s))\,dt$, where $U_t:=F(Y_t)$ is uniformly on $(0,1)$ distributed for each $t\in[0,1]$. We, therefore, recall in this section results for copula processes established in \citet{aulfaho11}.

Let $\bfY=(Y_t)_{t\in[0,1]}$ be a continuous stochastic process with
identical continuous marginal df $F$. Set
\[
\bfU=(U_t)_{t\in[0,1]}=(F(Y_t))_{t\in[0,1]},
\]
which is the \textit{copula process} corresponding to $\bfY$.

We conclude from \citet{dehal01} that the
process $\bfY$ is in the domain of attraction  of a MSP if, and
only if each $Y_t$ is in the domain of attraction of a univariate
extreme value distribution together with the condition that the
copula process converges in distribution to a standard MSP
$\bfeta$, that is
\[
\left(\max_{1\le i\le n}n(U^{(i)}_t-1)\right)_{t\in[0,1]}\to_D \bm{\eta}
\]
in $C[0,1]$, where $\bfU^{(i)}$, $i\in\N$, are independent copies
of $\bfU$.  Note that the univariate margins determine the norming
constants, so  the norming functions can be chosen as the constant
functions $a_n=1/n$, $b_n=1$, $n\in\N$. Lemma \ref{lem:weak_convergence_implies_domain_of_attraction} implies that $\bfU$ is in the functional domain of attraction of $\bfeta$.

Suppose that the rv $(Y_{i/d})_{i=1}^d$ is in the ordinary domain of attraction of a multivariate EVD (see, for instance, \citet[Section 5.2]{fahure10}). Then we know from \citet{aulbf11} that the copula $C_d$ corresponding to the rv $(Y_{i/d})_{i=1}^d$  satisfies the equation
\begin{equation}\label{eqn:expansion_of_copula}
C_d(\bfy)=1-\norm{\bfone-\bfy}_{D_d}+o\left(\norm{\bfone-\bfy}_{\infty}\right),
\end{equation}
as $\norm{\bfone-\bfy}_{\infty}\to\bfzero$, uniformly in $\bfy\in[0,1]^d$, where the $D$-norm is given by
\[
\norm{\bfx}_{D_d} =E\left(\max_{1\le i\le d}\left(\abs{x_i}Z_{i/d}\right)\right),\qquad \bfx\in\R^d.
\]

The following analogous result for the functional domain of attraction was established in \citet{aulfaho11}.

\begin{prop} Suppose that $\bfU$ with realizations in $\bar C^+[0,1]$ is a copula process.
The following equivalences hold:
\begin{align}
&\bfU\in \mathcal{D}(\bfeta) \mbox{ in the sense of condition \eqref{cond:definition_of_domain_of_attraction}}\nonumber\\
&\iff P\left(\bfU-1\le \frac f n\right)=1-\norm{\frac f n}_D+o\left(\frac 1 n \right),\quad f\in\barE,\mbox{ as }n\to\infty, \nonumber\\
&\iff P(\bfU-1\le \abs c f)=1+c\norm{f}_D+o(c), \quad f\in\barE,\mbox{ as }c\uparrow 0, \label{eqn:equivalent_formulation_of_domain_of_attraction}
\end{align}
\end{prop}

Note that condition \eqref{eqn:equivalent_formulation_of_domain_of_attraction} holds if
\begin{equation}\label{eqn:sharpening_of_equivalent_formulation_of_domain_of_attraction}
P(\bfU-1\le g)=1-\norm g_D+o(\norm g_\infty)\tag{\ref{eqn:equivalent_formulation_of_domain_of_attraction}'}
\end{equation}
as $\norm g_\infty\to 0$, uniformly for all $g\in\barE$ with
$\norm g_\infty\le 1$. It is an open problem whether \eqref{eqn:sharpening_of_equivalent_formulation_of_domain_of_attraction} and \eqref{eqn:equivalent_formulation_of_domain_of_attraction} are, actually, equivalent conditions.

\section{Sojourn Times and the Fragility
Index}\label{sec:sojourn_times_and_the_fragility_index}

Let $\bfY=(Y_t)_{t\in[0,1]}$ be a continuous stochastic process with
identical continuous marginal df $F$. We investigate in this section the
mean of the  sojourn time  of $\bfY$ above a threshold $s$
\[
S(s) =\int_0^1 1(Y_t>s)\,dt,
\]
under the condition that there is an exceedance, i.e.,
$S(s)>0$. In particular we establish its asymptotic equality with
the limit of the FI corresponding to $(Y_{i/n})_{1\le i\le n}$.

Before we present the main results of this section we need some
auxiliary results. Put for $n\in\N$
\[
S_n(s):=\frac 1n \sum_{i=1}^n 1(Y_{i/n}>s),
\]
which is a Riemann sum of the integral $S(s)$. We have
\[
S_n(s)\to_{n\to\infty} S(s)
\]
and, thus,
\[
P(S_n(s)\le x)\to_{n\to\infty} P(S(s)\le x)
\]
for each $x\ge 0$ such that $P(S(s)=x)=0$. As a consequence we obtain
\begin{align*}
P(S_n(s)\le x\mid S_n(s)>0) &= \frac{P(0<S_n(s)\le x)}{P(S_n(s)>0)}\\
&\to_{n\to\infty} \frac{P(0<S(s)\le x)}{P(S(s)>0)}\\
&= P(S(s)\le x\mid S(s)>0)
\end{align*}
for each such $x > 0$. This conclusion requires the following argument.

\begin{lemma} We have
\[
P(S_n(s)=0)\to_{n\to\infty} P(S(s)=0),
\]
\end{lemma}

\begin{proof}
We have
\[
P(S_n(s)=0)\le P(S_n(s)\le \varepsilon)\to_{n\to\infty}P(S(s)\le \varepsilon)=P(S(s)=0)+\delta,
\]
where $\varepsilon,\delta > 0$ can be made arbitrarily small. This
implies $\limsup_{n\to\infty}P(S_n(s)=0)\le P(S(s)=0)$. We have,
on the other hand,
\[
P(S(s)=0)=P\left(\bigcap_{n\in\N}\set{S_n(s)=0}\right)\le \liminf_{n\to\infty}P(S_n(s)=0),
\]
which implies the assertion.
\end{proof}

We have
\begin{align*}
S_n(s)&=\frac1n \sum_{i=1}^n 1(F(Y_{i/n})>F(s))\\
&=\frac 1n \sum_{i=1}^n 1(U_{i/n}>c)
\end{align*}
almost surely, where $c:=F(s)$.

Note that
\begin{align*}
FI_n(s)&:=E(nS_n(s)\mid S_n(s)>0)\\
&=E\left(\sum_{i=1}^n 1(U_{i/n} > c)\mid S_n(s)>0\right)\\
&=\sum_{i=1}^n P\left(U_{i/n}> c\mid S_n(s)>0\right)\\
&= \sum_{i=1}^n \frac{P (U_{i/n}> c)}{P\left(S_n(s)>0\right)}\\
&=n \frac{1-c}{1- P\left(S_n(s)=0\right)}
\end{align*}
is the FI of level $s$ corresponding to $Y_{i/n}$, $1\le i\le n$.
For an extensive investigation and extension of the FI we refer to
\citet{falt10a}. The following theorem is the
first main result of this section.

\begin{theorem}\label{thm:limit_of_FI_part_I}
Let $\bfY$ be a stochastic process in $C[0,1]$ with identical
continuous marginal df $F$. Suppose that the copula process
$\bfU=\left(F(Y_t)\right)_{t\in[0,1]}$  corresponding to  $\bfY$
is in the functional domain of attraction of a MSP $\bfeta$ with
generator constant $m\geq1$ as in Proposition
\ref{prop:characterization_of_EP}. Then we have
\[
\lim_{n\to\infty}\lim_{s\nearrow\omega(F)}\frac{FI_n(s)}n= \lim_{s\nearrow\omega(F)}
\lim_{n\to\infty}\frac{FI_n(s)}n =\lim_{s\nearrow\omega(F)}E(S(s)\mid
S(s)>0)= \frac 1m .
\]
\end{theorem}

\begin{proof}
Expansion \eqref{eqn:expansion_of_copula} implies for $n\in\N$
\begin{align*}
&P(S_n(s)>0)\\
&=1-P\left(\sum_{i=1}^n 1(U_{i/n}>c)=0\right)\\
&=1-P(U_{i/n}\le c,\,1\le i\le n)\\
&=1-C_n(c,\dots,c)\\
&=(1-c)\norm{(1,\dots,1)}_{D_n} + o\left((1-c)\norm{(1,\dots,1)}_{D_n}\right)\\
&=(1-c) E\left(\max_{1\le i\le n} Z_{i/n} \right) + o\left((1-c) E\left(\max_{1\le i\le n} Z_{i/n} \right)\right)
\end{align*}
as $c\uparrow 1$ and, thus,
\begin{align*}
\frac{FI_n(s)}n&=\frac{1-c}{P(S_n(s)>0)}\\
&=\frac 1{E\left(\max_{1\le i\le n} Z_{i/n} \right) + o\left( E\left(\max_{1\le i\le n} Z_{i/n}\right)\right)}
\end{align*}
as $c\uparrow 1$. We, thus, obtain
\begin{align*}
\lim_{n\to\infty}\lim_{s\nearrow\omega(F)} \frac{FI_n(s)}n &= \lim_{n\to\infty}\frac 1 {E\left(\max_{1\le i\le n} Z_{i/n} \right)} =\frac 1 {E\left(\max_{0\le t\le 1}Z_t\right)} = \frac 1m.
\end{align*}
We have, on the other hand,
$$
\lim_{n\to\infty}\frac{FI_n(s)}n= \lim_{n\to\infty} \frac{1-c}{1-
P\left(S_n(s)=0\right)}=\frac{1-c}{1- P\left(S(s)=0\right)}.
$$
Since $\bfU\in\mathcal{D}(\bfeta)$, we obtain from the equivalent
condition \eqref{eqn:equivalent_formulation_of_domain_of_attraction}
\begin{align*}
\lim_{s\nearrow\omega(F)}\lim_{n\to\infty}\frac{FI_n(s)}n&= \lim_{s\nearrow\omega(F)}\frac{1-c}{1- P\left(S(s)=0\right)}\\
&=\lim_{s\nearrow\omega(F)}\frac{1-c}{1- P\left(\bfY\leq s\right)}\\
&=\lim_{s\nearrow\omega(F)}\frac{1-c}{1- P\left(\bfU\leq c\right)}\\
&=\lim_{s\nearrow\omega(F)}\frac{1-c}{1- (1-(1-c)\norm{1}_D+o(1-c))}\\
&= \frac{1}{\norm{1}_D}\\
&= \frac{1}{E\left(\max_{0\le t\le 1}Z_t\right)}\\
&=\frac1m,
\end{align*}
where  $1$ is the constant function on $[0,1]$. Moreover, by the dominated convergence theorem
\begin{align*}
\frac{FI_n(s)}n&=E(S_n(s)\mid S_n(s)>0)\\
&=\frac{E(S_n(s))}{P(S_n(s)>0)}\\
&\to_{n\to\infty}\frac{E(S(s))}{P(S(s)>0)}\\
&=E(S(s)\mid S(s)>0).
\end{align*}

\end{proof}

\begin{rem}\label{rem:the_generator_constant_is_uniquely_determined}\upshape
While the generator $\bfZ$ of a standard MSP $\bfeta$ is in general not uniquely determined, the generator constant $m=E\left(\max_{t\in[0,1]}Z_t\right)=\norm 1_D$ is.
\end{rem}

\begin{rem}\upshape
Under the conditions of Theorem \ref{thm:limit_of_FI_part_I} we
have
$$
P(S(s)>0)= (1-c)m+o(1-c)\quad \textrm{as }c\nearrow1\quad \textrm{and}\quad E(S(s)) = 1-F(s).
$$
\end{rem}

To apply the preceding result to generalized Pareto processes defined below,  we add an extension of Theorem \ref{thm:limit_of_FI_part_I}. It is
shown by repeating the preceding arguments.

We call a copula process $\bfU=(U_t)_{t\in[0,1]}$ (upper)
\textit{tail continuous},  if the process
$\bfU_{c_0}:=\left(\max(c_0,U_t)\right)_{t\in[0,1]}$ is a.s.
continuous for some $c_0<1$. Note that in this case $\bfU_c$ is
a.s. continuous for each $c\ge c_0$.

A stochastic process $\bfY=(Y_t)_{t\in[0,1]}$ is said to have
ultimately  identical and continuous marginal df $F_t$, $t\in
[0,1]$, if $F_t(x)=F_s(x)$, $0\le s,t\le 1$, $x\ge x_0$ with
$F_1(x_0)<1$, and $F_1(x)$ is continuous for $x\ge x_0$.

\begin{theorem}\label{thm:limit_of_FI_part_II}
Let $\bfY=(Y_t)_{t\in[0,1]}$ be a stochastic process with
ultimately identical and continuous marginal df. Suppose that the
copula process pertaining to $\bfY$ is tail continuous and that it
is in the functional domain of attraction of a MSP $\bfeta$, whose finite
dimensional marginal distributions are given by
\[
G_{t_1,\dots,t_d}(\bfx)=\exp\left(-E\left(\max_{1\le i\le
d}\abs{x_i}Z_{t_i}\right)\right),
\]
$0\le t_1<\dots<t_d\le 1$, $\bfx\le\bfzero\in\R^d$, $d\in\N$. We
require that the stochastic process $\bfZ=(Z_t)_{t\in[0,1]}$ is
a.s. continuous and that its components satisfy $0\le Z_t\le m$
a.s., $E(Z_t)=1$, $t\in[0,1]$, for some $m\ge 1$. Then we have
\begin{align*}
\lim_{n\to\infty}\lim_{s\nearrow\omega(F)}\frac{FI_n(s)}n&=\lim_{s\nearrow\omega(F)} \lim_{n\to\infty}\frac{FI_n(s)}n\\
&=\lim_{s\nearrow\omega(F)}E(S(s)\mid S(s)>0)\\
&=\frac 1{E\left(\max_{0\le t\le 1}Z_t\right)}.
\end{align*}
 \end{theorem}

 \begin{exam}\upshape
 Consider the $d$-dimensional EVD $G(\bfx)=\exp(-\norm{\bfx}_p)$, $\bfx\le\bfzero\in\R^d$, $d\ge 2$, where the $D$-norm is the usual $p$-norm $\norm{\bfx}_D= \left(\sum_{i=1}^d\abs{x_i}^p\right)^{1/p}$ $=\norm{\bfx}_p$, $\bfx\in\R^d$, with $1\le p\le \infty$. This is known as the Gumbel-Hougaard or logistic model. The case $p=\infty$ yields the maximum-norm $\norm{\bfx}_\infty$. Let the rv $(Z_1,\dots,Z_d)$ be a generator of $\norm\cdot_p$, i.e., $0\le Z_i\le c$ a.s., $E(Z_i)=1$, $1\le i\le d$ with some $c\ge 1$, and $\norm{\bfx}_p= E\left(\max_{1\le i\le d}(\abs{x_i}Z_i)\right)$, $\bfx\in\R^d$. The rv $(Z_1,\dots,Z_d)$ can be extended by linear interpolation to a generator $\bfZ=(Z_t)_{t\in[0,1]}$ of a standard MSP $\bfeta$: Put for $i=1,\dots,d-1$
 \[
 Z_{(1-\vartheta)\frac{i-1}{d-1}+\vartheta\frac i{d-1}}:= (1-\vartheta)Z_{i-1}+\vartheta Z_i,\qquad 0\le\vartheta\le 1,
 \]
 which yields a continuous generator $\bfZ=(Z_t)_{t\in[0,1]}$. In this case we have
 \[
 \frac 1{E\left(\max_{0\le t\le 1}Z_t\right)}=\frac 1 {E\left(\max_{1\le i \le d}Z_i\right)} =\frac 1{\norm{(1,\dots,1)}_p} = \frac 1{d^{1/p}},
 \]
 i.e., the generator constant is $d^{1/p}$.

 Note that  a standard MSP $\bfeta$, whose finite dimensional marginal distributions $G_{t_1,\dots,t_d}$ are for each set $0\le t_1<t_2<\dots<t_d\le 1$ and each $d\ge 1$  given by $G_{t_1,\dots,t_d}(\bfx)=\exp(-\norm{\bfx}_p)$, $\bfx\le 0\in\R^d$, does not exist for $p\in[1,\infty)$. This follows from the fact that in this case the generator constant would be infinite. In case $p=\infty$, which is the case of complete dependence, one can choose $\bfeta=(\eta_t)_{t\in[0,1]}$ with $\eta_t:=\eta$, $t\in [0,1]$, where $\eta$ is a rv with standard negative exponential distribution. As a generator one can choose the constant function $Z_t=1$, $t\in[0,1]$.

 \end{exam}

 \begin{exam}[Generalized Pareto Process (GPP)]\label{exam:generalized_Pareto_process}\upshape
Let $\bfZ=(Z_t)_{t\in[0,1]} $ in $\bar C^+[0,1]$ with $0\le Z_t\le m$
a.s., $E(Z_t)=1$, $t\in[0,1]$, for some $m\ge 1$, and let $U$ be a
rv that is uniformly on $(0,1)$ distributed and which is
independent of $\bfZ$. Then the process
\[
\bfY:=\frac 1 U \bfZ
\]
with values in $\bar C^+[0,1]$ is an example of a \textit{generalized Pareto process (GPP)} (cf. \citet{buihz08}), as its univariate margins are (in its upper tails) standard
Pareto distributions:
\begin{align*}
F_t(x)&=P(Z_t\le x U)\\
&=\int_0^m P\left(\frac z x < U\right)\,(P*Z_t)(dz)\\
&=1-\frac 1 x E(Z_t)\\
&=1-\frac 1 x,\qquad x\ge m,\,t\in[0,1].
\end{align*}
Here $P*Z_t$ denotes the probability measure induced by $Z_t$, i.e., $(P*Z_t)(B)=P(Z_t\in B)$ for each $B$ in the Borel-$\sigma$ field of $\R$.

We have, moreover, by Fubini's theorem for all $f\in\barE$ with $\norm f_\infty\le 1/m$
\begin{align*}
P\left(-\frac 1\bfY\le f\right)&=1-\norm f_D,
\end{align*}
i.e., the
GPP $\bfV:=\left(\max(-1/Y_t,M)\right)_{0\le t\le 1}= \left(\max(-U/Z_t,M)\right)_{0\le t\le 1}$, with an arbitrary constant $M<0$, has the property that its df is in its upper
tail equal to
\begin{align*}
W(f)&:=P\left(\bfV\le f\right)  = 1+\log(G(f)),\quad f\in \barE,\,\norm f_\infty\le 1/m,
\end{align*}
where $G(f)=P(\bfeta\le f)$ is the df of the MSP $\bfeta$ with
$D$-norm $\norm\cdot_D$ and generator $\bfZ$ (cf. \citet[Section 4]{aulfaho11}).

The preceding representation of the upper tail of the df of a
GPP $\bfV$ in terms of $1+\log(G)$ is in complete accordance with
the unit- and multivariate case (see, for example, \citet[Chapter
5]{fahure10}).

We call in general a stochastic process $\bfV$ in $\bar C^-[0,1]$ a \textit{standard} GPP, if there is $\varepsilon_0>0,\, M<0$ with $P(\bfV\le f)=P\left(\left(\max(-U/Z_t,M)\right)_{0\le t\le 1}\le f\right)$ for all $f\in\bar E^-[0,1]$ with $\norm f_\infty\le \varepsilon_0$. As $Z_t$ may attain the value zero, we introduce the constant $M$ to ensure finite values of the process.

Note that the copula process pertaining to the GPP $\bfZ/U$ is in its upper tail given by the shifted standard GPP $1+\bfV$, which satisfies the
conditions of Theorem \ref{thm:limit_of_FI_part_II}. We, therefore, obtain for the GPP process $\bfZ/U$
\begin{align*}
\lim_{n\to\infty} \lim_{s\nearrow\omega(F)} \frac{FI_n(s)} n &= \lim_{s\nearrow\omega(F)} E(S(s)\mid S(s)>0)\ =\ \frac 1 {E\left(\max_{0\le t\le 1}Z_t\right)}.
\end{align*}
\end{exam}

The mathematical tools from Section \ref{sec:domain_of_attraction_for_copula_processes} enable also the computation of the (cumulative) expected shortfall corresponding to a stochastic process as defined below.

Let $\bfY=(Y_t)_{t\in[0,1]}$ be a stochastic process in $C[0,1]$ with identical and continuous univariate marginal df $F$ and put
\[
I(s)=\int_0^1 (Y_t-s)1(Y_t>s)\,dt.
\]
The number $I(s)$ can be interpreted as the total sum of excesses above the threshold $s$. The \textit{expected shortfall} at level $s$ pertaining to $\bfY$ is the total sum of excesses, given that there is at least one:
\[
\mathrm{ES}(s):=E(I(s)\mid S(s)>0).
\]

\begin{lemma}\label{lem:expansion_of_expected_shortfall}
Let $\bfU=(U_t)_{t\in[0,1]}=(F(Y_t))_{t\in[0,1]}$ be the copula process pertaining to $\bfY$. Then we have
\[
\mathrm{ES}(s)=\frac{\int_s^\infty 1-F(x)\,dx}{1-P\left(\sup_{t\in[0,1]}U_t\le F(s)\right)}.
\]
\end{lemma}

\begin{proof}
We have
\begin{align*}
E(I(s)\mid S(s)>0)&= E\left(\int_0^1 (Y_t-s)1(Y_t>s)\,dt\mid \int_0^1 1(Y_t>s)\,dt>0\right)\\
&= E\left(\int_0^1 (Y_t-s)1(Y_t>s)\,dt\mid \sup_{t\in[0,1]}Y_t>s\right)\\
&= \frac{E\left(\left(\int_0^1 (Y_t-s)1(Y_t>s)\,dt \right) 1\left(\sup_{t\in[0,1]}Y_t>s\right)\right)} {P\left(\sup_{t\in[0,1]}Y_t>s\right)}\\
&= \frac{E\left(\int_0^1 (Y_t-s)1(Y_t>s)\,dt \right)} {P\left(\sup_{t\in[0,1]}Y_t>s\right)},
\end{align*}
where by Fubini's theorem
\begin{align*}
E\left(\int_0^1 (Y_t-s)1(Y_t>s)\,dt \right)&=\int_0^1 E((Y_t-s)1(Y_t>s))\,dt\\
&=\int_0^1 \int_0^\infty 1- P(Y_t-s\le x)\,dx\,dt\\
&=\int_0^1 \int_0^\infty1-F(x+s)\,dx\,dt\\
&=\int_s^\infty 1-F(x)\,dx
\end{align*}
and
\begin{align*}
P\left(\sup_{t\in[0,1]}Y_t>s\right)&=1- P\left(\sup_{t\in[0,1]}Y_t\le s\right)\\
&=1- P\left(\sup_{t\in[0,1]}U_t\le F(s)\right).
\end{align*}
\end{proof}

Suppose in addition that the copula process $\bfU$ is in the domain of attraction in the sense of condition \eqref{eqn:equivalent_formulation_of_domain_of_attraction} of a standard MSP with generator constant $m$. Then there exists a $D$-norm $\norm\cdot_D$ on $C[0,1]$ with $\norm 1_D=m$ such that
\[
P(U_t\le F(s),\,t\in[0,1])=1-(1-F(s))\norm 1_D+o(1-F(s))
\]
as $s\nearrow \omega(F)$. The next result is, therefore, an obvious consequence of Lemma \ref{lem:expansion_of_expected_shortfall}.

\begin{prop}\label{prop:expansion_of_expected_shortfall}
If in addition the copula process $\bfU$ is in the domain of attraction of a standard MSP with generator constant $m$, then we obtain
\[
\mathrm{ES}(s)=\frac{\int_s^\infty 1-F(x)\,dx}{1-F(s)}\left(\frac 1m +o(1-F(s))\right)
\]
as $s\nearrow \omega(F)$.
\end{prop}

Proposition \ref{prop:expansion_of_expected_shortfall} precisely separates the contribution of the dependence structure of the stochastic process $\bfY$ on the expected shortfall as the threshold increases, which is $1/\norm 1_D=1/m$, from that of the marginal distribution, which is the first factor. In particular we obtain that the expected shortfall converges in $[0,\infty)$ as $s\nearrow\omega(F)$ if and only if $\lim_{s\nearrow\omega(F)}\int_s^{\omega(F)}1-F(t)\,dt/(1-F(s)):=c\in [0,\infty)$. And in this case its limit is $c/\norm{\bfone}_D$.

\section{Sojourn Time Distribution}\label{sec: sojourn_time_distribution}
In this section we compute the asymptotic sojourn time distribution of such processes, which are in a certain neighborhood of a standard GPP. A standard MSP is a prominent example. In this setup we can replace the constant threshold s by a threshold function.

The sojourn time distribution of a standard GPP is easily computed as the following lemma shows. This distribution is independent of the threshold level $s$, which reveals another exceedance stability of a GPP. Note that we replace the constant threshold line $s$ in what follows by a \textit{threshold function} $sf(t)$, where $f\in \barE$ is fixed and $s$ is the variable \textit{threshold level}.

\begin{lemma}\label{lem:sojourn_time_distribution_of_standard_gpp}
Let $\bfV$ in $\bar C^-[0,1]$ be a standard GPP, i.e. there is an $\varepsilon_0>0$ such that $P(\bfV\le g)=P(-U/\bfZ\le g)$ for all $g\in\bar E^-[0,1]$ with $\norm g_\infty\le \varepsilon_0$, where $U$ is uniformly on $(0,1)$
distributed and independent of the generator
$\bfZ=(Z_t)_{t\in[0,1]}$, which is continuous and satisfies $0\le
Z_t\le m$, $E(Z_t)=1$, $t\in [0,1]$, for some $m\ge 1$. Choose $f\in\barE$. Then there is an $s_0>0$ such that the sojourn time
df $H_f$  of $\bfV$ above $sf$ is given by
\begin{align*}
&P\left(\int_0^1 1\left(V_t>s f(t)\right)\,dt> y\mid \int_0^1
1\left(V_t>s f(t)\right)\,dt> 0\right)\\
&= \frac{\int_0^{m\norm f_\infty} P\left(\int_0^1 1\left(\abs{f(t)}Z_t> u\right)\,dt>
y\right)\,du}{\int_0^{m\norm f_\infty} P\left(\int_0^1 1\left(\abs{f(t)}Z_t> u\right)\,dt>
0\right)\,du}\\
&=:1-H_f(y),\qquad 0\le y\le 1,\,0< s\le s_0,
\end{align*}
provided the denominator is greater than zero. Note that $H_f(0)=0$, $H_f(1)=1$.
\end{lemma}

\begin{exam}\upshape
Any continuous df $F$ on $[0,1]$ can occur as a sojourn time df. Take $Z_t=1$, $0\le t\le 1$, which provides the case of complete dependence of the margins of the corresponding standard MSP $\bfeta$. Choose a continuous df $F:[0,1]\to[0,1]$ and put $f(t)=F(t)-1$, $0\le t\le 1$. Then the sojourn time df equals $F$, $H_f(y)=F(y)$, $y\in[0,1]$.

If we take, on the other hand, $f(t)=-1$, $t\in[0,1]$, then $H_f$ has all its mass at $1$, i.e., $H_f(y)=0$, $y<1$, and $H_f(1)=1$. These examples show in particular that the sojourn time df $H_f$ can be continuous as well as discrete.
\end{exam}

\begin{proof}
The assertion is an immediate consequence of standard rules of integration
together with conditioning on $U=u$:
\begin{align*}
&P\left(\int_0^1 1\left(V_t>s f(t)\right)\,dt> y\right)\\
&=P\left(\int_0^1 1\left(U<s\abs{f(t)}Z_t\right)\,dt> y\right)\\
&=\int_0^1 P\left(\int_0^1 1\left(u<s\abs{f(t)}Z_t\right)\,dt> y\right)\,
du,\\
\intertext{where substituting $u$ by $su$ yields} &=s
\int_0^{1/s}  P\left(\int_0^1 1\left(\abs{f(t)}Z_t> u\right)\,dt>
y\right)\,du\\
&=  s \int_0^{m\norm f_\infty}P\left(\int_0^1 1\left(\abs{f(t)}Z_t> u\right)\,dt>
y\right)\,du
\end{align*}
if $s\le 1/(m\norm f_\infty)$. This implies the assertion.
\end{proof}

Next we will extend the preceding lemma to processes $\bfxi$ in $\bar C^-[0,1]$ which are in certain neighborhoods of a standard GPP $\bfV$. Precisely, we require that for a given function $f\in\bar E^-_1[0,1]:=\set{f\in\barE:\,\norm f_\infty \le 1}$
\begin{equation}\label{eqn:upper_delta_neighborhood_condition}
P\left(\xi_{t_i}> s f(t_i),\,1\le i\le k\right)=P\left(V_{t_i}>s f(t_i),\,1\le i\le k\right) +o(s)
\end{equation}
for each set $0\le t_1<\dots<t_k\le 1$, $k\in\N$, and
\begin{equation}\label{eqn:lower_delta_neighborhood_condition}
P(\bfxi\le s f)=P(\bfV\le s f)+o(s)
\end{equation}
as $s\downarrow 0$.

An example of a process satisfying conditions \eqref{eqn:upper_delta_neighborhood_condition} and \eqref{eqn:lower_delta_neighborhood_condition} is a standard MSP $\bfeta$, which follows by Lemma \ref{lem:expansion_of_survivor_function_of_eta} below together with equation \eqref{eq:distribution_function_of_standard_EP}. The next lemma follows from elementary computations.

\begin{lemma}\label{cor:expansion_of_df_of_GPP}
For each standard GPP $\bfV$ there exists $s_0>0$ such that for $0\le s\le s_0$ and for each $f\in\barE$ with $\norm f_\infty\le 1$
\begin{itemize}
\item[(i)]
\[
P(\bfV\le sf)=1-sE\left(\max_{t\in[0,1]}\left(\abs{f(t)}Z_t\right)\right)=1-s\norm f_D,
\]
\item[(ii)]
\[
P(\bfV>sf)=sE\left(\min_{t\in[0,1]}\left(\abs{f(t)}Z_t\right)\right),
\]
\item[(iii)]
\[
P\left(V_{t_i}>sf(t_i),\,1\le i\le k\right)= s E\left(\min_{1\le i\le k}\left(\abs{f(t_i)}Z_{t_i}\right)\right)
\]
for each set $0\le t_1<\dots< t_k\le 1$, $k\in\N$.
\end{itemize}
\end{lemma}

The next result extends Lemma \ref{lem:sojourn_time_distribution_of_standard_gpp} to processes which satisfy condition \eqref{eqn:upper_delta_neighborhood_condition} and \eqref{eqn:lower_delta_neighborhood_condition}.

\begin{prop}\label{prop:sojourn_time_distribution_for_neighborhoods}
Suppose that $\bfxi\in\bar C^-[0,1]$ has identical univariate margins and that it satisfies condition \eqref{eqn:upper_delta_neighborhood_condition} as well as \eqref{eqn:lower_delta_neighborhood_condition}. Choose $f\in\bar E_1^-[0,1]$. Then the asymptotic sojourn time distribution of $\bfxi$, conditional on the assumption that it is positive,  is given by
\[
P\left(\int_0^1 1\left(\xi_t>sf(t)\right)\,dt>y\mid \int_0^1 1\left(\xi_t>sf(t)\right)\,dt>0\right)\to_{s\downarrow 0}1-H_f(y),
\]
where the sojourn time df $H_f$ is given in Lemma \ref{lem:sojourn_time_distribution_of_standard_gpp}.
\end{prop}

\begin{proof}
We establish this result by establishing convergence of characteristic functions. Put $I_s:=\int_0^1 1(\xi_t> s f(t))\,dt$, $s>0$. The characteristic function of the rv $I_s$, conditional on the event that it is positive, is
\[
E\left(\exp\left(itI_s\right) \mid I_s>0\right)=\frac{\int_{\set{I_s>0}}\exp(it I_s)\,dP}{P(I_s>0)}.
\]
Note that $0\le I_s\le 1$. By the dominated convergence theorem we have
\begin{align}\label{eqn:expansion_of_characteristic_function_of_sojourn_time}
\int_{\set{I_s>0}}\exp(it I_s)\,dP&= \int_{\set{I_s>0}}\sum_{k=0}^\infty \frac{(itI_s)^k}{k!} \,dP\nonumber\\
&=\sum_{k=0}^\infty \frac{(it)^k}{k!}  \int_{\set{I_s>0}} I_s^k\,dP\nonumber\\
&=P(I_s>0) + \sum_{k=1}^\infty \frac{(it)^k}{k!}  \int_\Omega I_s^k\,dP\nonumber\\
&=P(I_s>0) + \sum_{k=1}^\infty \frac{(it)^k}{k!} E\left(I_s^k\right).
\end{align}

From condition \eqref{eqn:lower_delta_neighborhood_condition} we obtain
\begin{align}\label{eqn:expansion_of_probability_that_I_s_< 0}
P(I_s>0)&=1-P(I_s=0)\nonumber\\
&=1-P(\bfxi\le sf)\nonumber\\
&=1-P(\bfV\le sf)+o(s)\nonumber\\
&=s\left(E\left(\max_{t\in[0,1]}\abs{f(t)Z_t}\right)+o(1)\right)
\end{align}
as $s\downarrow 0$.

From Fubini's theorem we obtain for $k\in\N$
\begin{align*}
E(I_s^k)&=E\left(\left(\int_0^1 1(\xi_t> s f(t))\,dt\right)^k\right)\\
&=E\left(\int_0^1\dots\int_0^1 \prod_{i=1}^k 1\left(\xi_{t_i}> s f(t_i)\right)\,dt_1\dots dt_k\right)\\
&=\int_0^1\dots\int_0^1 E\left( \prod_{i=1}^k 1\left(\xi_{t_i}> s f(t_i)\right)\right)\,dt_1\dots dt_k\\
&=\int_0^1\dots\int_0^1 P\left(\xi_{t_i}>s f(t_i),\,1\le i\le k\right) \,dt_1\dots dt_k.
\end{align*}
We have by condition \eqref{eqn:upper_delta_neighborhood_condition}
\[
P\left(\xi_{t_i}>s f(t_i),\,1\le i\le k\right)\le P(\xi_{t_1}>-s)=P(\xi_0>-s)=P(V_0>-s)+o(s)
\]
uniformly for $t_1,\dots,t_k\in[0,1]$ and, thus, $P\left(\xi_{t_i}>s f(t_i),\,1\le i\le k\right)/s$ is uniformly bounded. Condition \eqref{eqn:upper_delta_neighborhood_condition} together with the dominated convergence theorem now implies
\begin{align}\label{eqn:expansion_of_E(I_s^k)}
\frac{E(I_s^k)} s &= \int_0^1\dots\int_0^1 \frac{P\left(\xi_{t_i}>s f(t_i),\,1\le i\le k\right)}s \,dt_1\dots dt_k \nonumber\\
&\to_{s\downarrow 0}\int_0^1\dots\int_0^1 E\left(\min_{1\le i\le k}\abs{f(t_i)Z_{t_i}}\right)  \,dt_1\dots dt_k.
\end{align}

From equations \eqref{eqn:expansion_of_characteristic_function_of_sojourn_time}-\eqref{eqn:expansion_of_E(I_s^k)} we obtain
\begin{align*}
&\int_{\set{I_s>0}}\exp(itI_s)\,dP\\
&=s(1+o(1))\left(E\left(\max_{t\in[0,1]}\abs{f(t)Z_t}\right)\right.\\
&\hspace*{1cm} + \left.\sum_{k=1}^n \frac{(it)^k}{k!} \left(\int_0^1\dots\int_0^1 E\left(\min_{1\le i\le k} \abs{f(t_i)Z_{t_i}}\right)\,dt_1\dots dt_k\right)\right)\\
&\hspace*{.5cm} + \sum_{k=n+1}^\infty \frac{(it)^k}{k!} E(I_s^k),
\end{align*}
where $n\in\N$ is chosen such that for a given $\varepsilon>0$ we have $\sum_{k=m+1}^\infty 1/k!\le\varepsilon$. As $I_s\in[0,1]$, we obtain
\begin{align*}
E(I_s^k)&\le E(I_s)\\
&=\int_0^1 P(\xi_t>s f(t))\,dt\\
&=\int_0^1 P(\xi_0> s f(t))\,dt\\
&\le P\left(\xi_0>s \inf_{t\in[0,1]}f(t)\right)\\
&= s \inf_{t\in[0,1]}\abs{f(t)}+o(s)
\end{align*}
by condition \eqref{eqn:upper_delta_neighborhood_condition} and, thus,
\begin{align*}
&\int_{\set{I_s>0}}\exp(itI_s)\,dP\\
&=s(1+o(1))\left(E\left(\max_{t\in[0,1]}\abs{f(t)Z_t}\right)\right.\\
&\hspace*{1cm} + \left.\sum_{k=1}^n \frac{(it)^k}{k!} \left(\int_0^1\dots\int_0^1 E\left(\min_{1\le i\le k} \abs{f(t_i)Z_{t_i}}\right)\,dt_1\dots dt_k\right)+O(\varepsilon)\right)\\
\end{align*}
as $s\downarrow 0$. Since $\varepsilon>0$ was arbitrary we obtain
\begin{align*}
&\lim_{s\downarrow 0}\frac{\int_{\set{I_s>0}}\exp(itI_s)\,dP}{P(I_s>0)}\\
&=1+ \frac{\sum_{k=1}^\infty \frac{(it)^k}{k!} \left(\int_0^1\dots\int_0^1 E\left(\min_{1\le i\le k} \abs{f(t_i)Z_{t_i}}\right)\,dt_1\dots dt_k\right)} {E\left(\max_{t\in[0,1]}\abs{f(t)Z_t}\right)}\\
&=:\varphi(t),\qquad t\in\R.
\end{align*}
An inspection of the preceding arguments shows that $\varphi$ is the characteristic function of the sojourn time df $H_f$, which completes the proof.
\end{proof}

We conclude this section by showing that a standard MSP $\bfeta$ satisfies condition \eqref{eqn:upper_delta_neighborhood_condition} and, thus, Proposition \ref{prop:sojourn_time_distribution_for_neighborhoods} applies. Note that condition \eqref{eqn:lower_delta_neighborhood_condition} follows from \eqref{eq:distribution_function_of_standard_EP} and Taylor expansion of $\exp$.

\begin{lemma}\label{lem:expansion_of_survivor_function_of_eta}
Let $\bfeta$ be a standard MSP with generator $\bfZ$. Then we
obtain for $f\in\bar E^-[0,1]$
\[
P\left(\eta_{t_i}> s f(t_i),\,1\le i\le k\right) = sE\left(\min_{1\le i\le k}\left(\abs{f(t_i)}Z_{t_i}\right)\right) + o(s)
\]
as $s\downarrow 0$ for any set $0\le t_1<\dots<t_k\le 1$, $k\in\N$.
\end{lemma}

\begin{proof} The inclusion-exclusion theorem yields
\begin{align*}
&P\left(\bigcap_{i=1}^k\set{\eta_{t_i}>sf(t_i)}\right)\\
&=1- P\left(\bigcup_{i=1}^k\set{\eta_{t_i}\le sf(t_i)}\right)\\
&=1-\sum_{\emptyset\not= T\subset\set{1,\dots,k}} (-1)^{\abs T - 1} P\left(\bigcap_{i\in T}\set{\eta_{t_i}\le sf(t_i)}\right)\\
&=1-\sum_{\emptyset\not= T\subset\set{1,\dots,k}} (-1)^{\abs T - 1} \exp\left(- s E\left( \max_{i\in T}\left(\abs{f(t_i)}Z_{t_i}\right)\right)\right)\\
&=:1-H(s)\\
&=H(0)-H(s),
\end{align*}
where the function $H$ is differentiable and, thus,
\begin{align*}
\lim_{s\downarrow 0}\frac{P\left(\eta_{t_i}> s f(t_i),\,1\le i\le k\right)}s&= -\lim_{s\downarrow 0}\frac{H(s)-H(0)}s\\
&=-H'(0)\\
&= \sum_{\emptyset\not= T\subset\set{1,\dots,k}} (-1)^{\abs T - 1}   E\left( \max_{i\in T}\left(\abs{f(t_i)}Z_{t_i}\right)\right)\\
&= E\left( \min_{i\in T}\left(\abs{f(t_i)}Z_{t_i}\right)\right),
\end{align*}
since $\sum_{\emptyset\not= T\subset\set{1,\dots,k}} (-1)^{\abs T - 1} \max_{i\in T}a_i=\min_{1\le i\le k}a_i$ for arbitrary numbers $a_1,\dots,a_k\in\R$, which can be seen by induction.
\end{proof}

\section{Excursion Time}\label{sec:cluster_length}
The considerations in the previous section enable us also to compute the limit distribution of the excursion time above the threshold $sf$ of a process $\bfX$ in $\bar C^-[0,1]$, which is in a neighborhood of a standard GPP. Precisely, we require the following condition. Choose $0\le a\le b\le 1$, and denote by $\bar C^-[a,b]$  the set of continuous functions $f:[a,b]\to (-\infty,0]$. We suppose that for $f\in \bar C^-[a,b]$
\begin{equation}\label{cond:neighborhood_of_gpp}
P\left(X_t>sf(t),\,t\in[a,b]\right)=P\left(V_t>sf(t),\,t\in[a,b]\right)+o(s)
\end{equation}
as $s\downarrow 0$, where $\bfV=(V_t)_{t\in[0,1]}$ is a standard GPP. Note that
\begin{equation}\label{eqn:condition_for_excursion_time}
P(V_t>sf(t),\,t\in[a,b])=s E\left(\min_{a\le t\le b}\left(\abs{f(t)}Z_t\right)\right) + o(s),\quad s\in(0,s_0),
\end{equation}
and that we allow the case $a=b$. We do not require $\bfX$ to have identical marginal distributions.

A standard MSP $\bfeta$ satisfies condition \eqref{cond:neighborhood_of_gpp}, see \citet{aulfaho11}. Another example is the following class of processes. Substitute the rv $U$ in the GPP $\bfV=(\max(-U/Z_t,M))_{t\in[0,1]}$ by a rv $W\ge 0$, which is independent of $\bfZ$ as well and whose df $H$ satisfies
\[
H(x)=x+o(x),\qquad \mbox{as }x\to 0.
\]
The standard exponential df $F(x)=1-\exp(-x)$, $x>0$, is a typical example. Then the process
\[
\bfX:=\left(\max\left(-\frac W {Z_t},M\right)\right)_{t\in[0,1]}
\]
satisfies condition \eqref{eqn:condition_for_excursion_time} as well.

The \textit{remaining excursion time} above $sf$ of the process $\bfX$ with inspection point $t_0\in[0,1)$ is  the remaining time that the process spends above $sf$, under the condition that $X_{t_0}>s f(t_0)$, i.e., it is defined by
\[
\tau_{t_0}(s):=\sup\set{L\in(0,1-t_0]:\, X_t>s f(t), t\in[t_0,t_0+L)}
\]
under the condition that $X_{t_0}>s f(t_0)$.

\begin{prop}
Suppose that $\bfX$ in $\bar C^-[0,1]$ satisfies condition \eqref{cond:neighborhood_of_gpp}. Then we have for $u\in[0,1-t_0)$ and $f\in\bar C^-[a,b]$ with $f(t_0)<0$
\[
\lim_{s\downarrow 0}P\left(\tau_{t_0}(s)> u\mid X_{t_0}> sf(t_0)\right) =\frac{E\left(\min_{t_0\le t\le u}(\abs{f(t)}Z_t)\right)}{\abs{f(t_0)}}.
\]
\end{prop}

\begin{proof}
We have for $u\in[0,1-t_0)$
\begin{align*}
P\left(\tau_{t_0}(s)> u\mid X_{t_0}> sf(t_0)\right) &= \frac{P(X_t>sf(t),\,t\in[t_0,t_0+u])}{P\left(X_{t_0}>s f(t_0)\right)}\\
&= \frac{P(V_t>sf(t),\,t\in[t_0,t_0+u])+o(s)}{P\left(V_{t_0}>s f(t_0)\right)+o(s)}\\
&= \frac{E\left(\min_{t_0\le t\le t_0+ u}(\abs{f(t)}Z_t)\right)}{\abs{f(t_0)}} + o(1)
\end{align*}
as $s\downarrow 0$.
\end{proof}

The \textit{asymptotic} remaining excursion time $T_{t_0}$, as $s\downarrow 0$, with inspection point $t_0\in[0,1)$ has, consequently, the continuous df
\[
P\left(T_{t_0}\le u\right)=1- \frac{E\left(\min_{t_0\le t\le t_0+ u}(\abs{f(t)}Z_t)\right)}{\abs{f(t_0)}}
\]
for $0\le u<1-t_0$, and possibly positive mass at $1-t_0$:
\[
P\left(T_{t_0}=1-t_0\right)=\frac{E\left(\min_{t_0\le t\le 1}(\abs{f(t)}Z_t)\right)}{\abs{f(t_0)}}.
\]
Its expected value is, therefore, given by
\begin{align*}
E\left(T_{t_0}\right)&=\int_0^{1-t_0}P\left(T_{t_0}>u\right)\,du\\
&=\frac 1{\abs{f(t_0)}}\int_0^{1-t_0} E\left(\min_{t_0\le t\le t_0+ u}(\abs{f(t)}Z_t)\right)\,du\\
&= \frac 1{\abs{f(t_0)}}E\left(\int_{t_0}^1 \min_{t_0\le t\le u}(\abs{f(t)}Z_t)\,du\right).
\end{align*}

\section*{Acknowledgment}
The authors are grateful to two anonymous reviewers and the associate editor for their careful reading of the manuscript and their constructive remarks, from which the paper has benefitted a lot.


\begin{thebibliography}{12}
\providecommand{\natexlab}[1]{#1}
\providecommand{\url}[1]{\texttt{#1}}
\providecommand{\urlprefix}{URL }
\providecommand{\selectlanguage}[1]{\relax}
\providecommand{\bibinfo}[2]{#2}
\providecommand{\href}[2]{#2}
\providecommand{\eprint}[2][]{\href{#1}{#2}}

\bibitem[{Aulbach et~al.(2011{\natexlab{\textit{a}}})Aulbach, Bayer, and
  Falk}]{aulbf11}
\bibinfo{author}{\textsc{Aulbach, S.}}, \bibinfo{author}{\textsc{Bayer, V.}},
  and \bibinfo{author}{\textsc{Falk, M.}}
  (\bibinfo{year}{2011}{\natexlab{\textit{a}}}).
\newblock \bibinfo{title}{A multivariate piecing-together approach with an
  application to operational loss data}.
\newblock \textit{\bibinfo{journal}{Bernoulli}}, \bibinfo{volume}{to appear}.

\bibitem[{Aulbach et~al.(2011{\natexlab{\textit{b}}})Aulbach, Falk, and
  Hofmann}]{aulfaho11}
\bibinfo{author}{\textsc{Aulbach, S.}}, \bibinfo{author}{\textsc{Falk, M.}},
  and \bibinfo{author}{\textsc{Hofmann, M.}}
  (\bibinfo{year}{2011}{\natexlab{\textit{b}}}).
\newblock \bibinfo{title}{On extreme value processes and the functional
  $D$-norm}.
\newblock \bibinfo{type}{Tech. Rep.}, \bibinfo{institution}{University of
  W\"{u}rzburg}.
\newblock \bibinfo{note}{Submitted},
  \eprint[http://arxiv.org/abs/1107.5136]{{\ttfamily arXiv:1107.5136
  [math.PR]}}.

\bibitem[{Berman(1992)}]{berm92}
\bibinfo{author}{\textsc{Berman, S.~M.}} (\bibinfo{year}{1992}).
\newblock \textit{\bibinfo{title}{Sojourns and Extremes of Stochastic
  Processes}}.
\newblock Wadsworth \& Brooks/Cole Statistics/Probability Series.
  \bibinfo{publisher}{Chapman and Hall/CRC}, \bibinfo{address}{Boca Raton}.

\bibitem[{Buishand et~al.(2008)Buishand, de~Haan, and Zhou}]{buihz08}
\bibinfo{author}{\textsc{Buishand, T.~A.}}, \bibinfo{author}{\textsc{de~Haan,
  L.}}, and \bibinfo{author}{\textsc{Zhou, C.}} (\bibinfo{year}{2008}).
\newblock \bibinfo{title}{On spatial extremes: With application to a rainfall
  problem}.
\newblock \textit{\bibinfo{journal}{Ann. Appl. Stat.}}
  \textbf{\bibinfo{volume}{2}}, \bibinfo{pages}{624--642}.

\bibitem[{Falk et~al.(2010)Falk, H\"{u}sler, and Reiss}]{fahure10}
\bibinfo{author}{\textsc{Falk, M.}}, \bibinfo{author}{\textsc{H\"{u}sler, J.}},
  and \bibinfo{author}{\textsc{Reiss, R.-D.}} (\bibinfo{year}{2010}).
\newblock \textit{\bibinfo{title}{Laws of Small Numbers: Extremes and Rare
  Events}}.
\newblock \bibinfo{edition}{3rd} ed. \bibinfo{publisher}{Birkh\"{a}user},
  \bibinfo{address}{Basel}.

\bibitem[{Falk and Tichy(2012)}]{falt10a}
\bibinfo{author}{\textsc{Falk, M.}}, and \bibinfo{author}{\textsc{Tichy, D.}}
  (\bibinfo{year}{2012}).
\newblock \bibinfo{title}{Asymptotic conditional distribution of exceedance
  counts}.
\newblock \textit{\bibinfo{journal}{Adv. Appl. Prob.}}
  \textbf{\bibinfo{volume}{44}}.

\bibitem[{Geluk et~al.(2007)Geluk, de~Haan, and de~Vries}]{gelhv07}
\bibinfo{author}{\textsc{Geluk, J.~L.}}, \bibinfo{author}{\textsc{de~Haan,
  L.}}, and \bibinfo{author}{\textsc{de~Vries, C.~G.}} (\bibinfo{year}{2007}).
\newblock \bibinfo{title}{Weak \& Strong Financial Fragility}.
\newblock \bibinfo{howpublished}{Tinbergen Institute Discussion Paper. TI
  2007-023/2}.

\bibitem[{Gin\'{e} et~al.(1990)Gin\'{e}, Hahn, and Vatan}]{ginhv90}
\bibinfo{author}{\textsc{Gin\'{e}, E.}}, \bibinfo{author}{\textsc{Hahn, M.}},
  and \bibinfo{author}{\textsc{Vatan, P.}} (\bibinfo{year}{1990}).
\newblock \bibinfo{title}{Max-infinitely divisible and max-stable sample
  continuous processes}.
\newblock \textit{\bibinfo{journal}{Probab. Th. Rel. Fields}}
  \textbf{\bibinfo{volume}{87}}, \bibinfo{pages}{139--165}.

\bibitem[{de~Haan and Ferreira(2006)}]{dehaf06}
\bibinfo{author}{\textsc{de~Haan, L.}}, and \bibinfo{author}{\textsc{Ferreira,
  A.}} (\bibinfo{year}{2006}).
\newblock \textit{\bibinfo{title}{Extreme Value Theory: An Introduction}}.
\newblock Springer Series in Operations Research and Financial Engineering.
  \bibinfo{publisher}{Springer}, \bibinfo{address}{New York}.

\bibitem[{de~Haan and Lin(2001)}]{dehal01}
\bibinfo{author}{\textsc{de~Haan, L.}}, and \bibinfo{author}{\textsc{Lin, T.}}
  (\bibinfo{year}{2001}).
\newblock \bibinfo{title}{On convergence toward an extreme value distribution
  in $C[0,1]$}.
\newblock \textit{\bibinfo{journal}{Ann. Probab.}}
  \textbf{\bibinfo{volume}{29}}, \bibinfo{pages}{467--483}.

\bibitem[{Hsing and Leadbetter(1998)}]{hsile98}
\bibinfo{author}{\textsc{Hsing, T.}}, and \bibinfo{author}{\textsc{Leadbetter,
  M.~R.}} (\bibinfo{year}{1998}).
\newblock \bibinfo{title}{On the excursion random measure of stationary
  processes}.
\newblock \textit{\bibinfo{journal}{Ann. Probab.}}
  \textbf{\bibinfo{volume}{26}}, \bibinfo{pages}{710--742}.

\bibitem[{Molchanov(2005)}]{mol05}
\bibinfo{author}{\textsc{Molchanov, I.}} (\bibinfo{year}{2005}).
\newblock \textit{\bibinfo{title}{Theory of Random Sets}}.
\newblock Probability and Its Applications. \bibinfo{publisher}{Springer},
  \bibinfo{address}{London}.

\end{thebibliography}

\end{document}